\documentclass[11pt,twoside,letterpaper]{article}

\usepackage{times,fancyhdr}
\usepackage[dvips]{graphicx}
\usepackage[english]{babel}
\usepackage{hyperref}
\usepackage{amsmath}
\usepackage{amssymb}
\usepackage{amsthm}
\usepackage{graphicx}
\usepackage{newlfont}
\usepackage{setspace}
\usepackage{subfig}
\usepackage{color}
\usepackage{mathtools}
\usepackage{relsize}

\usepackage{url}
\usepackage{fancybox}
\usepackage[Lenny]{fncychap}
\usepackage{multirow}

\newtheorem{theorem}{Theorem}[section]
\newtheorem{proposition}[theorem]{Proposition}
\newtheorem{definition}[theorem]{Definition}
\newtheorem{example}[theorem]{Example}
\newtheorem{remark}[theorem]{Remark}


\makeatletter
\def\cleardoublepage{\clearpage\if@twoside \ifodd\c@page\else%
    \hbox{}%
    \thispagestyle{empty}%
    \newpage%
    \if@twocolumn\hbox{}\newpage\fi\fi\fi}
\makeatother

\def\tablename{Table}
\makeatletter
\renewcommand{\fnum@table}[1]{\tablename~\thetable.}
\makeatother

\begin{document}

\title{
{\begin{flushleft}
\vskip 0.45in
\end{flushleft}
\vskip 0.45in
\bfseries\scshape Optimal control and numerical software: an overview}}
\author{\bfseries Helena Sofia Rodrigues$^1$
\and \bfseries M. Teresa T. Monteiro$^2$
\and \bfseries Delfim F. M. Torres$^3$
\thanks{E-mail addresses: sofiarodrigues@esce.ipvc.pt, tm@dps.uminho.pt, delfim@ua.pt}\\
$^1$CIDMA and School of Business,\\Viana do Castelo Polytechnic Institute, Portugal\\
$^2$R\&D Centre Algoritmi, Department of Production and Systems,\\University of Minho, Portugal\\
$^3$CIDMA, Department of Mathematics,\\University of Aveiro, Portugal}

\date{}


\maketitle


\thispagestyle{empty}
\setcounter{page}{1}
\thispagestyle{fancy}
\fancyhead{}
\fancyhead[L]{This is a preprint of a paper whose final and definite form will appear in the book\\ 
'Systems Theory: Perspectives, Applications and Developments',\\
Nova Science Publishers, Editor: Francisco Miranda.\\
Submitted 21/Sept/2013; Revised 01/Dec/2013; Accepted 22/Jan/2014.}
\fancyhead[R]{}
\fancyfoot{}
\renewcommand{\headrulewidth}{0pt}


\begin{abstract}
Optimal Control (OC) is the process of determining control
and state trajectories for a dynamic system, over a period
of time, in order to optimize a given performance index.
With the increasing of variables and complexity, OC problems can no longer be solved
analytically and, consequently, numerical methods are required. For this purpose,
direct and indirect methods are used. Direct methods consist in the discretization of the OC problem,
reducing it to a nonlinear constrained optimization problem. Indirect methods are based
on the Pontryagin Maximum Principle, which in turn reduces to a boundary value problem.
In order to have a more reliable solution, one can solve the same problem through
different approaches. Here, as an illustrative example, an epidemiological application
related to the rubella disease is solved using several software packages,
such as the routine ode45 of Matlab, OC-ODE, DOTcvp toolbox, IPOPT and Snopt,
showing the state of the art of numerical software for OC.
\end{abstract}


\vspace{.08in} \noindent 


\section{Introduction}

Historically, optimal control (OC) is an extension of the calculus of
variations. The first formal results of the calculus of variations
can be found in the seventeenth century.
Johann Bernoulli challenged other famous contemporary mathematicians---such
as Newton, Leibniz, Jacob Bernoulli, L'H\^{o}pital and von Tschirnhaus---with the
\emph{brachistochrone} problem: ``if a particle moves, under the
influence of gravity, which path between two fixed points enables
the trip of shortest time?'' (see, e.g., \cite{Bryson1996}).
This and other specific problems were solved, and a general mathematical
theory was developed by Euler and Lagrange. The most fruitful
applications of the calculus of variations have been to
theoretical physics, particularly in connection with Hamilton's
principle or the Principle of Least Action. Early applications to
economics appeared in the late 1920s and early 1930s by Ross,
Evans, Hottelling and Ramsey, with further applications published
occasionally thereafter \cite{Sussmann1997}.

\pagestyle{fancy}
\fancyhead{}
\fancyhead[EC]{H. S. Rodrigues, M. T. T. Monteiro, D. F. M. Torres}
\fancyhead[EL,OR]{\thepage}
\fancyhead[OC]{Optimal control and numerical software: an overview}
\fancyfoot{}
\renewcommand\headrulewidth{0.5pt}

The generalization of the calculus of variations to optimal
control theory was strongly motivated by military applications and has
developed rapidly since 1950. The decisive breakthrough was achieved
by the Russian mathematician Lev S. Pontryagin (1908-1988) and his
co-workers (V. G. Boltyanskii, R. V. Gamkrelidz and E. F.
Misshchenko) with the formulation and proof of the Pontryagin Maximum
Principle \cite{Pontryagin1962}. This principle has provided research with
suitable conditions for optimization problems with differential
equations as constraints. The Russian team generalized variational problems by
separating control and state variables and admitting control
constraints. The two approaches use a different point of view
and the OC approach often affords insight into a
problem that might be less readily apparent through the calculus
of variations. OC is also applied to problems where
the calculus of variations is not convenient, such as those
involving constraints on the derivatives of functions \cite{Leitmann1997}.

The theory of OC brought new approaches to Mathematics with Dynamic Programming.
Introduced by R. E. Bellman, Dynamic Programming makes use of the principle
of optimality and it is suitable for solving discrete problems, allowing for a significant
reduction in the computation of the optimal controls (see \cite{Kirk1998}). It is also possible
to obtain a continuous approach to the principle of optimality that leads to the solution
of a partial differential equation called the Hamilton-Jacobi-Bellman equation.
This result allowed to bring new connections between the OC problem and the Lyapunov stability theory.

Before the computer age, only fairly simple OC problems could be solved.
The arrival of the computer enabled the application of OC theory and its methods
to many complex problems. Selected examples are as follows:
\begin{itemize}
\item Physical systems, such as stable performance of motors and
machinery, robotics, and optimal guidance of rockets \cite{Goh2008,Molavi2008};

\item Aerospace, including driven problems, orbits transfers, development of
satellite launchers and recoverable problems of atmospheric
reentry \cite{Bonnard2008,Hermant2010};

\item Economics and management, such as optimal exploitation of natural
resources, energy policies, optimal investment of production
strategies \cite{Munteanu2008,Sun2008};

\item Biology and medicine, as regulation of physiological functions, plants growth,
infectious diseases, oncology, radiotherapy \cite{Joshi2002,Joshi2006,Lenhart2007,Nanda2007}.
\end{itemize}

Nowadays, OC is an extensive theory with several approaches.
One can adjust controls in a system
to achieve a certain goal, where the underlying system can include:
ordinary differential equations, partial differential equations,
discrete equations, stochastic differential equations,
integro-difference equations, combination of discrete and
continuous systems. In this work we restrict ourselves to deterministic
OC theory of ordinary differential equations in a fixed time interval.


\section{Optimal control problems}
\label{sec:1:1}

A typical OC problem requires a performance index or cost functional,
$J[x(\cdot),u(\cdot)]$; a set of state variables, $x(\cdot) \in X$;
and a set of control variables, $u(\cdot) \in U$.
The main goal consists in finding a piecewise continuous control $u(t)$,
$t_0 \leq t \leq t_f$, and the associated state variable $x(t)$,
to maximize the given objective functional.

\begin{definition}[Basic OC Problem in Lagrange form]
\label{OC_Lagrange}
\index{Lagrange form}
\noindent An OC problem is:
\begin{equation}
\label{OC_Lagrange_eq}
\begin{gathered}
\underset{u(\cdot)}{\max} \  J[x(\cdot),u(\cdot)]
=\int_{t_0}^{t_f}f(t,x(t),u(t))dt,\\
\text{s.t. }\  \dot{x}(t)=g(t,x(t),u(t)),\\
x(t_0)=x_0.
\end{gathered}
\end{equation}
\end{definition}

\begin{remark}
\label{rem:1}
The value of $x(t_f)$ in \eqref{OC_Lagrange_eq} is free,
which means that the value of $x(t_f)$ is unrestricted.
Sometimes one also considers problems with $x(t_f)$ fixed,
\emph{i.e}, $x(t_f)=x_f$ for a certain given $x_f$.
\end{remark}

For our purposes, $f$ and $g$ will always be continuously differentiable
functions in all three arguments.  The controls will always
be piecewise continuous, and the associated states will always
be piecewise differentiable. Note that we can switch back and forth
between maximization and minimization,
by simply negating the cost functional:
$$
\min\{J\}=-\max\{-J\}.
$$

An OC problem can be presented in many different, but equivalent ways,
depending on the purpose or the software to be used.


\subsection{Lagrange, Mayer and Bolza formulations}
\label{sec:1:2}

There are three well known equivalent formulations to describe an
OC problem, which are the Lagrange (Definition~\ref{OC_Lagrange}),
Mayer and Bolza forms \cite{Chachuat2007,Zabczyk2008}.

\begin{definition}[Basic OC Problem in Bolza form]
\label{OC_Bolza}\index{Bolza form}
Bolza's formulation of the OC problem is:
\begin{equation}
\label{OC_Bolza_eq}
\begin{gathered}
\underset{u(\cdot)}{\max} \ J[x(\cdot),u(\cdot)]
=\phi\left(t_{0},x(t_{0}),t_f,x(t_f)\right)+\int_{t_0}^{t_f}f(t,x(t),u(t))dt,\\
\text{s.t. } \ \dot{x}(t)=g(t,x(t),u(t)),\\
x(t_0)=x_0,
\end{gathered}
\end{equation}
where $\phi$ is a continuously differentiable function.
\end{definition}

\begin{definition}[Basic OC Problem in Mayer form]
\label{OC_Mayer}\index{Mayer form}
Mayer's formulation of the OC problem is:
\begin{equation}
\label{OC_Mayer_eq}
\begin{gathered}
\underset{u(\cdot)}{\max} \ J[x(\cdot),u(\cdot)]
=\phi\left(t_{0},x(t_{0}),t_f,x(t_f)\right),\\
\text{s.t. } \ \dot{x}(t)=g(t,x(t),u(t)),\\
x(t_0)=x_0.
\end{gathered}
\end{equation}
\end{definition}

\begin{theorem}
\label{thm:1}
The three formulations, Lagrange (Definition~\ref{OC_Lagrange})\index{Lagrange form},
Bolza (Definition~\ref{OC_Bolza})\index{Bolza form} and Mayer
(Definition~\ref{OC_Mayer}\index{Mayer form}), are equivalent.
\end{theorem}

\begin{proof}
See, e.g., \cite{Chachuat2007,Zabczyk2008}.
\end{proof}

The proof of Theorem~\ref{thm:1} gives a method to rewrite an optimal control problem
in any of the three forms to any other of the three forms.
Note that, from a computational perspective,
some of the OC problems, often presented in the Lagrange form,
should be converted into the equivalent Mayer form. Hence, using a standard procedure,
one rewrites the cost functional, augmenting the state vector with an extra component
(\textrm{cf.}, e.g., \cite{Lewis1995}). More precisely,
the Lagrange formulation \eqref{OC_Lagrange_eq} is rewritten as
\begin{equation}
\label{cap1_lagrange-mayer}
\begin{gathered}
\underset{u(\cdot)}{\max} \ x_c(t_f),\\
\text{s.t. } \ \dot{x}(t)=g(t,x(t),u(t)),\\
\dot{x}_c(t)=f(t,x(t),u(t)),\\
x(t_0)=x_0,\\
x_c(t_0)=0.
\end{gathered}
\end{equation}


\subsection{Pontryagin's Maximum Principle}
\label{sec:1:3}
\index{Pontryagin's Maximum Principle}

Necessary first order optimality conditions were developed by
Pontryagin and his co-workers. The result is considered as one of the most important
results of Mathematics in the 20th century.
Pontryagin introduced the idea of adjoint functions to append the
differential equation to the objective functional. Adjoint
functions have a similar purpose as Lagrange multipliers in
multivariate calculus, which append constraints to the function of
several variables to be maximized or minimized.

\begin{definition}[Hamiltonian]
\label{Hamiltonian}\index{Hamiltonian}
Consider the OC problem \eqref{OC_Lagrange_eq}. The function
\begin{equation}
\label{eq:H}
H(t, x, u,\lambda)=f(t,x,u)+\lambda \, g(t,x,u)
\end{equation}
is called the Hamiltonian (function), and $\lambda$
is the adjoint variable.
\end{definition}

We are ready to formulate the Pontryagin Maximum Principle (PMP)
for problem \eqref{OC_Lagrange_eq}.

\begin{theorem}[Pontryagin's Maximum Principle for \eqref{OC_Lagrange_eq}]
\label{PMP}\index{Pontryagin's Maximum Principle}
If $u^{*}(\cdot)$ and $x^{*}(\cdot)$ are optimal for problem \eqref{OC_Lagrange_eq},
then there exists a piecewise differentiable adjoint variable
$\lambda(\cdot)$ such that
$$
H(t,x^{*}(t),u(t),\lambda(t))
\leq H(t,x^{*}(t),u^{*}(t),\lambda(t))
$$
for all controls $u(t)$ at each time $t$, where $H$ is the
Hamiltonian \eqref{eq:H}, and
\begin{equation*}
\begin{split}
\lambda'(t) &=-\frac{\partial H(t,x^{*}(t),u^{*}(t),\lambda(t))}{\partial x},\\
\lambda(t_f)&=0.
\end{split}
\end{equation*}
\end{theorem}

The proof of Theorem~\ref{PMP} follows classical variational methods
and can be found, e.g., in \cite{Lenhart2007}. The original
Pontryagin's book \cite{Pontryagin1962}
or Clarke's book \cite{Clarke1990} are good references
to find more general results and their detailed proofs.

\begin{remark}
The last condition of Theorem~\ref{PMP}, $\lambda(t_f)=0$, is
called the transversality condition\index{Transversality condition},
and is only used when the OC problem does not have the terminal value in the state variable,
\emph{i.e.}, $x(t_f)$ is free (cf. Remark~\ref{rem:1}).
\end{remark}

Theorem~\ref{PMP} converts the problem of finding a control which
maximizes the objective functional subject to the state ODE and
initial condition, into the problem of optimizing the Hamiltonian
pointwise. As a consequence, we have
\begin{equation}
\label{Hu}
\displaystyle{\frac{\partial H}{\partial u}}=0
\end{equation}
at $u^{*}$ for each $t$, that is, the Hamiltonian\index{Hamiltonian} has a
critical point at at $u^{*}$. Usually this condition is called the \emph{optimality
condition}\index{Optimality condition}.

\begin{remark}
If the Hamiltonian\index{Hamiltonian} is linear in the control variable $u$, it can
be difficult to calculate $u^{*}$ from the optimality equation,
since $\frac{\partial H}{\partial u}$ would not contain $u$.
Specific ways to solve such kind of problems can be found, for example,
in \cite{Lenhart2007}.
\end{remark}

Until here we have shown necessary conditions to solve basic optimal control problems.
Now, it is important to study some conditions that can guarantee the existence
of a finite objective functional value at the optimal control and state variables
\cite{Fleming1975,Kamien1991,Lenhart2007,Macki1982}.
The following is an example of a sufficient condition.

\begin{theorem}
\label{thm:suf}
Consider the following problem:
\begin{equation*}
\begin{gathered}
\underset{u(\cdot)}{\max}\  J\left[x(\cdot),u(\cdot)\right]
=\int_{t_0}^{t_f}f(t,x(t),u(t))dt,\\
\text{s.t. } \ \dot{x}(t)=g(t,x(t),u(t)),\\
x(t_0)=x_0.
\end{gathered}
\end{equation*}
Suppose that $f(t,x,u)$ and $g(t,x,u)$ are both continuously differentiable
functions in their three arguments and concave in $x$ and $u$. If $u^{*}(\cdot)$
is a control with associated state $x^{*}(\cdot)$ and $\lambda(\cdot)$
a piecewise differentiable function such that $u^{*}(\cdot)$, $x^{*}(\cdot)$
and $\lambda(\cdot)$ together satisfy
\begin{gather*}
f_u+\lambda g_u=0 \Leftrightarrow \frac{\partial H}{\partial u} = 0,\\
\lambda'=-(f_x+\lambda g_x) \Leftrightarrow \lambda'= - \frac{\partial H}{\partial x},\\
\lambda(t_f)=0,\\
\lambda(t)\geq 0
\end{gather*}
on $t_0\leq t \leq t_f$, then
$$
J[x^{*}(\cdot),u^{*}(\cdot)] \geq J[x(\cdot),u(\cdot)]
$$
for any admissible pair $\left(x(\cdot),u(\cdot)\right)$.
\end{theorem}

\begin{proof}
The proof of this theorem can be found in \cite{Lenhart2007}.
\end{proof}

Theorem~\ref{thm:suf} is not strong enough to guarantee
that $J[x^{*}(\cdot),u^{*}(\cdot)]$ is finite.
Such results usually require some conditions on $f$ and/or $g$.
Next theorem is an example of an existence result from \cite{Fleming1975}
(cf. \cite[Theorem~2.2]{Lenhart2007}).

\begin{theorem}
\label{thm:ex}
Let the set of controls for problem \eqref{OC_Lagrange_eq}
be Lebesgue integrable functions on $t_0\leq t\leq t_f$ in $\mathbb{R}$.
Suppose that $f(t,x,u)$ is concave in $u$, and there exist constants
$C_1, C_2, C_3 >0$, $C_4$, and $\beta >1$ such that
\begin{gather*}
g(t,x,u)=\alpha(t,x)+\beta(t,x)u,\\
|g(t,x,u)|\leq C_1(1+|x|+|u|),\\
|g(t,x_1,u)-g(t,x,u)|\leq C_2 |x_1-x|(1+|u|),\\
f(t,x,u)\leq C_3|u|^{\beta}-C_4
\end{gather*}
for all $t$ with $t_0\leq t\leq t_1$, $x$, $x_1$, $u$ in $\mathbb{R}$.
Then there exists an optimal pair $\left(x^{*}(\cdot),u^{*}(\cdot)\right)$
maximizing $J$, with $J[x^{*}(\cdot),u^{*}(\cdot)]$ finite.
\end{theorem}

\begin{proof}
The proof is given in \cite{Fleming1975}.
\end{proof}

\begin{remark}
For a minimization problem, $f$ would have a convex property
and the inequality on $f$ would be reversed (coercivity).
\end{remark}

It is important to note that the necessary conditions developed to this point deal
with piecewise continuous optimal controls, while the existence
Theorem~\ref{thm:ex} guarantees an optimal control which is only Lebesgue integrable.
This gap can be overcome by studying regularity conditions
\cite{my:reg1,my:reg2}.


\subsection{Optimal control with bounded controls}
\label{sec:1:4}

Many problems, to be realistic, require bounds on the controls.

\begin{definition}[OC problem with bounded controls]
An OC problem with bounded control, in Lagrange form, is:
\begin{equation}
\label{OC_bounded_control}
\begin{gathered}
\underset{u(\cdot)}{\max} J[x(\cdot),u(\cdot)]=\int_{t_0}^{t_f} f(t,x(t),u(t)) dt,\\
\text{s.t. }\ \dot{x}(t)=g(t,x(t),u(t)),\\
x(t_0)=x_0,\\
a\leq u(t)\leq b,
\end{gathered}
\end{equation}
where $a$ and $b$ are fixed real constants with $a<b$.
\end{definition}

The Pontryagin Maximum Principle (Theorem~\ref{PMP})
remains valid for problems with bounds on the control,
except the maximization is over all admissible controls,
that is, $a\leq u(t)\leq b$ for all $t \in [t_0, t_f]$.

\begin{theorem}[Pontryagin's Maximum Principle for \eqref{OC_bounded_control}]
\label{thm:PMP:bc}
If $u^{*}(\cdot)$ and $x^{*}(\cdot)$ are optimal for problem \eqref{OC_bounded_control},
then there exists a piecewise differentiable adjoint variable $\lambda(\cdot)$ such that
$$
H(t,x^{*}(t),u(t),\lambda(t))\leq H(t,x^{*}(t),u^{*}(t),\lambda(t))
$$
for all admissible controls $u$ at each time $t$, where $H$ is the Hamiltonian \eqref{eq:H}, and
\begin{gather}
\lambda'(t)=-\frac{\partial H(t,x^{*}(t),u^{*}(t),\lambda(t))}{\partial x} \tag{adjoint condition},\\
\lambda(t_f)=0 \tag{transversality condition}\index{Transversality condition}.
\end{gather}
\end{theorem}

The following proposition is a direct consequence of Theorem~\ref{thm:PMP:bc}.
The proof can be found, e.g., in \cite{Kamien1991} or \cite{Lenhart2007}.

\begin{proposition}
The optimal control $u^{*}(\cdot)$ to problem \eqref{OC_bounded_control}
must satisfy the following optimality condition\index{Optimality condition}:
\begin{equation}
\label{eq:opt:cond}
u^{*}(t)=
\begin{cases}
a & \textrm{if } \frac{\partial H}{\partial u}<0\\
\tilde{u} & \textrm{if } \frac{\partial H}{\partial u}=0\\
b & \textrm{if } \frac{\partial H}{\partial u}>0,
\end{cases}
\end{equation}
where $a\leq  \tilde{u} \leq b$,
is obtained by the expression $\frac{\partial H}{\partial u}=0$.
In particular, the optimal control $u^{*}(\cdot)$ maximizes $H$ pointwise with respect to $a \leq u \leq b$.
\end{proposition}

\begin{remark}
If we have a minimization problem instead of maximization,
then $u^{*}$ is instead chosen to minimize $H$ pointwise.
This has the effect of reversing $<$ and $>$ in the first and third lines
of the optimality condition\index{Optimality condition} \eqref{eq:opt:cond}.
\end{remark}

So far, we have only examined problems with one control
and one state variable. Often,
it is necessary to consider more variables.
Below, one such optimal control problem, related to rubella, is presented.
The PMP continues valid for problems with several state and several
control variables.

\begin{example}
\label{example_rubeola}
Rubella, commonly known as German measles, is most common in child age,
caused by the rubella virus. Children recover more quickly than adults.
Rubella can be very serious during pregnancy. The virus is contracted through
the respiratory tract and has an incubation period of 2 to 3 weeks.
The primary symptom of rubella virus infection is the appearance
of a rash on the face which spreads to the trunk and limbs
and usually fades after three days. Other symptoms include low grade fever,
swollen glands, joint pains, headache and conjunctivitis.
We present an optimal control problem to study the dynamics
of rubella over three years, using a vaccination process ($u$)
as a measure to control the disease. More details can be found in \cite{Buonomo2011}.
Let $x_1$ represent the susceptible population, $x_2$ the proportion
of population that is in the incubation period, $x_3$ the proportion
of population that is infected with rubella, and $x_4$ the rule
that keeps the population constant. The optimal control problem can be defined as:
\begin{equation}
\label{cap1:ode_rubeola}
\begin{gathered}
\min \ \int_{0}^{3}(Ax_3+u^2)dt\\
\text{s.t. } \ \dot{x}_1=b-b(p x_2+q x_2)-b x_1-\beta x_1 x_3 - u x_1,\\
\dot{x}_2=b p x_2 +\beta x_1 x_3 -(e+b)x_2,\\
\dot{x}_3=e x_2-(g+b)x_3,\\
\dot{x}_4=b-b x_4,
\end{gathered}
\end{equation}
with initial conditions $x_1(0)=0.0555$, $x_2(0)=0.0003$,
$x_3(0)=0.0004$, $x_4(0)=1$ and the parameters $b=0.012$, $e=36.5$,
$g=30.417$, $p=0.65$, $q=0.65$, $\beta=527.59$ and $A=100$.
The control $u$ is defined as taking values in the interval $[0,0.9]$.
\end{example}

It is not easy to solve analytically the problem of Example~\ref{example_rubeola}.
For the majority of real OC applications, it is necessary to employ numerical methods.


\section{Numerical methods to solve optimal control problems}

In the last decades, the computational power has been developed in an amazing way.
Not only in hardware issues, such as efficiency, memory capacity, speed,
but also in terms of software robustness. Ground breaking achievements
in the field of numerical solution techniques for differential and integral
equations have enabled the simulation of highly complex real world scenarios.
OC also won with these improvements, and numerical methods
and algorithms have evolved significantly.


\subsection{Indirect methods}
\label{sec:2:2}
\index{Indirect methods}

In an indirect method, the PMP\index{Pontryagin's Maximum Principle}
is used. Therefore, the indirect approach leads to a multiple-point boundary-value
problem that is solved to determine candidate optimal trajectories, called extremals.
To apply it, it is necessary to explicitly get the adjoint equations,
the control equations, and all the transversality conditions\index{Transversality condition},
if they exist. A numerical approach using the indirect method,
known as the \emph{backward-forward sweep method}, is now presented.
This method is described in \cite{Lenhart2007}.
The process begins with an initial guess on the control variable. Then, simultaneously,
the state equations are solved forward in time and the adjoint
equations are solved backward in time. The control is updated
by inserting the new values of states and adjoints into its characterization, and
the process is repeated until convergence occurs.
Let us consider $\vec{x}=(x_1,\ldots,x_N+1)$ and $\vec{\lambda}=(\lambda_1,\ldots,\lambda_N+1)$
the vector of approximations for the state and the adjoint.
The main idea of the algorithm is described as follows.
\begin{description}
\item[Step 1.] Make an initial guess for $\vec{u}$
over the interval ($\vec{u}\equiv 0$ is almost always sufficient).
\item[Step 2.] Using the initial condition $x_1=x(t_0)=a$
and the values for $\vec{u}$, solve $\vec{x}$ forward
in time according to its differential equation in the optimality system.
\item[Step 3.] Using the transversality condition $\lambda_{N+1}=\lambda(t_f)=0$
and the values for $\vec{u}$ and $\vec{x}$, solve $\vec{\lambda}$ backward
in time according to its differential equation in the optimality system.
\item[Step 4.] Update $\vec{u}$ by entering the new $\vec{x}$ and $\vec{\lambda}$
values into the characterization of the optimal control.
\item[Step 5.] Verify convergence: if the variables are sufficiently close
to the corresponding ones in the previous iteration, then output the current
values as solutions, otherwise return to Step 2.
\end{description}
For Steps 2 and 3, Lenhart and Workman \cite{Lenhart2007} use, for the state
and adjoint systems, the Runge--Kutta fourth order
procedure to make the discretization process.
On the other hand, Wang \cite{Wang2009} applies the same philosophy
but solving the differential equations with the \texttt{ode45}\index{ode45 routine}
solver of \texttt{Matlab}\index{Matlab}. This solver is based on an explicit
Runge--Kutta (4,5) formula, the Dormand--Prince pair. That means
that the \texttt{ode45} numerical solver combines a fourth and a fifth order method,
both of which being similar to the classical fourth order Runge--Kutta method.
These vary the step size, choosing it at each step, in an attempt to achieve the desired
accuracy. Therefore, the \texttt{ode45} solver is suitable for a wide variety
of initial value problems in practical applications. In general, \texttt{ode45}
is the best method to apply, as a first attempt, to most problems \cite{Houcque}.

\begin{example}
\label{ex_rubeola_Lenhart}
Let us consider the problem of Example~\ref{example_rubeola} about rubella disease.
With $\vec{x}(t)=\left(x_1(t),x_2(t),x_3(t),x_4(t)\right)$
and $\vec{\lambda}(t)=\left(\lambda_1(t),\lambda_2(t),\lambda_3(t),\lambda_4(t)\right)$,
the Hamiltonian of this problem can be written as\index{Hamiltonian}
\begin{multline*}
\label{cap2:hamiltonian_rubeola}
H(t,\vec{x}(t),u(t),\vec{\lambda}(t))= Ax_3+u^2
+\lambda_1\left(b-b(p x_2+q x_2)-b x_1-\beta x_1 x_3 - u x_1\right)\\
+\lambda_2\left(b p x_2 +\beta x_1 x_3 -(e+b)x_2\right)
+\lambda_3\left(e x_2-(g+b)x_3\right)
+\lambda_4\left(b-b x_4\right).
\end{multline*}
Using the PMP, the optimal control problem can be studied with the control system
\begin{equation*}
\label{cap2:ode2_rubeola}
\begin{cases}
\dot{x}_1=b-b(p x_2+q x_2)-b x_1-\beta x_1 x_3 - u x_1\\
\dot{x}_2=b p x_2 +\beta x_1 x_3 -(e+b)x_2\\
\dot{x}_3=e x_2-(g+b)x_3\\
\dot{x}_4=b-b x_4
\end{cases}
\end{equation*}
subject to initial conditions $x_1(0)=0.0555$, $x_2(0)=0.0003$,
$x_3(0)=0.0004$, $x_4(0)=1$, and the adjoint system
\begin{equation*}
\label{cap2:ode3_rubeola}
\begin{cases}
\dot{\lambda}_1=\lambda_1(b+u+\beta x_3) - \lambda_2\beta x_3\\
\dot{\lambda}_2=\lambda_1 b p + \lambda_2(e+b+p b)-\lambda_3 e\\
\dot{\lambda}_3=-A+\lambda_1(b q +\beta x_1)-\lambda_2\beta x_1+\lambda_3(g+b)\\
\dot{\lambda}_4=\lambda_4 b
\end{cases}
\end{equation*}
with transversality conditions\index{Transversality condition}
$\lambda_i(3)=0$, $i=1,\ldots,4$. The optimal control is given by
\begin{equation*}
u^{*}=\left\{
\begin{array}{lll}
0 & \textrm{ if } & \frac{\partial H}{\partial u}<0,\\[0.25cm]
\frac{\lambda_1 x_1}{2} & \textrm{ if } & \frac{\partial H}{\partial u}=0,\\[0.25cm]
0.9 & \textrm{ if } & \frac{\partial H}{\partial u}>0.
\end{array}
\right.
\end{equation*}
We present here the main part of the code for the backward-forward
sweep method with fourth order Runge--Kutta. The complete
code can be found in the website \cite{SofiaSITE}.
The obtained optimal curves for the states variables
and optimal control are shown in Figure~\ref{cap2_rubeola}.
{\tiny{
\begin{verbatim}
for i = 1:M
    m11 = b-b*(p*x2(i)+q*x3(i))-b*x1(i)-beta*x1(i)*x3(i)-u(i)*x1(i);
    m12 = b*p*x2(i)+beta*x1(i)*x3(i)-(e+b)*x2(i);
    m13 = e*x2(i)-(g+b)*x3(i);
    m14 = b-b*x4(i);

    m21 = b-b*(p*(x2(i)+h2*m12)+q*(x3(i)+h2*m13))-b*(x1(i)+h2*m11)-...
        beta*(x1(i)+h2*m11)*(x3(i)+h2*m13)-(0.5*(u(i) + u(i+1)))*(x1(i)+h2*m11);
    m22 = b*p*(x2(i)+h2*m12)+beta*(x1(i)+h2*m11)*(x3(i)+h2*m13)-(e+b)*(x2(i)+h2*m12);
    m23 = e*(x2(i)+h2*m12)-(g+b)*(x3(i)+h2*m13);
    m24 = b-b*(x4(i)+h2*m14);

    m31 = b-b*(p*(x2(i)+h2*m22)+q*(x3(i)+h2*m23))-b*(x1(i)+h2*m21)-...
        beta*(x1(i)+h2*m21)*(x3(i)+h2*m23)-(0.5*(u(i) + u(i+1)))*(x1(i)+h2*m21);
    m32 = b*p*(x2(i)+h2*m22)+beta*(x1(i)+h2*m21)*(x3(i)+h2*m23)-(e+b)*(x2(i)+h2*m22);
    m33 = e*(x2(i)+h2*m22)-(g+b)*(x3(i)+h2*m23);
    m34 = b-b*(x4(i)+h2*m24);

    m41 = b-b*(p*(x2(i)+h2*m32)+q*(x3(i)+h2*m33))-b*(x1(i)+h2*m31)-...
        beta*(x1(i)+h2*m31)*(x3(i)+h2*m33)-u(i+1)*(x1(i)+h2*m31);
    m42 = b*p*(x2(i)+h2*m32)+beta*(x1(i)+h2*m31)*(x3(i)+h2*m33)-(e+b)*(x2(i)+h2*m32);
    m43 = e*(x2(i)+h2*m32)-(g+b)*(x3(i)+h2*m33);
    m44 = b-b*(x4(i)+h2*m34);

    x1(i+1) = x1(i) + (h/6)*(m11 + 2*m21 + 2*m31 + m41);
    x2(i+1) = x2(i) + (h/6)*(m12 + 2*m22 + 2*m32 + m42);
    x3(i+1) = x3(i) + (h/6)*(m13 + 2*m23 + 2*m33 + m43);
    x4(i+1) = x4(i) + (h/6)*(m14 + 2*m24 + 2*m34 + m44);
end

for i = 1:M
    j = M + 2 - i;

    n11 = lambda1(j)*(b+u(j)+beta*x3(j))-lambda2(j)*beta*x3(j);
    n12 = lambda1(j)*b*p+lambda2(j)*(e+b-p*b)-lambda3(j)*e;
    n13 = -A+lambda1(j)*(b*q+beta*x1(j))-lambda2(j)*beta*x1(j)+lambda3(j)*(g+b);
    n14 = b*lambda4(j);

    n21 = (lambda1(j) - h2*n11)*(b+u(j)+beta*(0.5*(x3(j)+x3(j-1))))-...
        (lambda2(j) - h2*n12)*beta*(0.5*(x3(j)+x3(j-1)));
    n22 = (lambda1(j) - h2*n11)*b*p+(lambda2(j) - h2*n12)*(e+b-p*b)-(lambda3(j) - h2*n13)*e;
    n23 = -A+(lambda1(j) - h2*n11)*(b*q+beta*(0.5*(x1(j)+x1(j-1))))-...
        (lambda2(j) - h2*n12)*beta*(0.5*(x1(j)+x1(j-1)))+(lambda3(j) - h2*n13)*(g+b);
    n24 = b*(lambda4(j) - h2*n14);

    n31 = (lambda1(j) - h2*n21)*(b+u(j)+beta*(0.5*(x3(j)+x3(j-1))))-...
        (lambda2(j) - h2*n22)*beta*(0.5*(x3(j)+x3(j-1)));
    n32 = (lambda1(j) - h2*n21)*b*p+(lambda2(j) - h2*n22)*(e+b-p*b)-(lambda3(j) - h2*n23)*e;
    n33 = -A+(lambda1(j) - h2*n21)*(b*q+beta*(0.5*(x1(j)+x1(j-1))))-...
        (lambda2(j) - h2*n22)*beta*(0.5*(x1(j)+x1(j-1)))+(lambda3(j) - h2*n23)*(g+b);
    n34 = b*(lambda4(j) - h2*n24);

    n41 = (lambda1(j) - h2*n31)*(b+u(j)+beta*x3(j-1))-(lambda2(j) - h2*n32)*beta*x3(j-1);
    n42 = (lambda1(j) - h2*n31)*b*p+(lambda2(j) - h2*n32)*(e+b-p*b)-(lambda3(j) - h2*n33)*e;
    n43 = -A+(lambda1(j) - h2*n31)*(b*q+beta*x1(j-1))-...
        (lambda2(j) - h2*n32)*beta*x1(j-1)+(lambda3(j) - h2*n33)*(g+b);
    n44 = b*(lambda4(j) - h2*n34);

    lambda1(j-1) = lambda1(j) - h/6*(n11 + 2*n21 + 2*n31 + n41);
    lambda2(j-1) = lambda2(j) - h/6*(n12 + 2*n22 + 2*n32 + n42);
    lambda3(j-1) = lambda3(j) - h/6*(n13 + 2*n23 + 2*n33 + n43);
    lambda4(j-1) = lambda4(j) - h/6*(n14 + 2*n24 + 2*n34 + n44);
end
u1 = min(0.9,max(0,lambda1.*x1/2));
\end{verbatim}}}
\begin{figure}[ptbh]
\center
\includegraphics[scale=0.46]{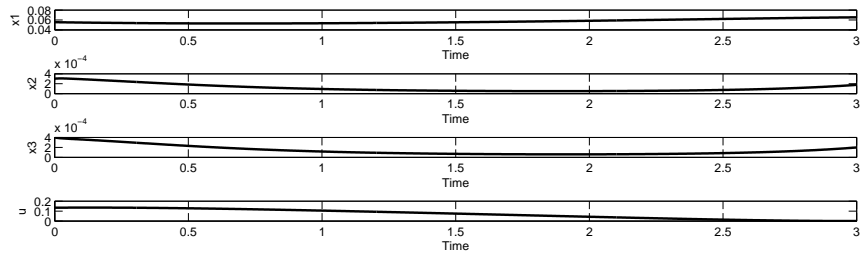}\\
{\caption{\label{cap2_rubeola} The optimal curves
for the rubella problem of Example~\ref{example_rubeola}.}}
\end{figure}
\end{example}
There are several difficulties to overcome when
an optimal control problem is solved by indirect methods\index{Indirect methods}.
Firstly, it is necessary to calculate the Hamiltonian\index{Hamiltonian},
adjoint equations, the optimality\index{Optimality condition}
and transversality conditions\index{Transversality condition}.
Besides, the approach is not flexible, since each time a new problem is formulated,
a new derivation is required. In contrast, a direct method\index{Direct method}
does not require explicit derivation of necessary conditions.
Due to its simplicity, the direct approach has been gaining popularity
in numerical optimal control over the past three decades \cite{Betts2001}.


\subsection{Direct methods}
\label{sec:2:3}

A new family of numerical methods for dynamic
optimization has emerged, referred as direct methods.\index{Direct method}
This development has been driven by the industrial need to solve large-scale
optimization problems and it has also been supported by the rapidly increasing
computational power. A direct method\index{Direct method} constructs a sequence of points
$x_1, x_2,\ldots,x^{*}$, such that the objective function $F$ to be minimized
satisfies $F(x_1)>F(x_2)> \cdots >F(x^{*})$. Here the state and/or control
are approximated using an appropriate function approximation
(e.g., polynomial approximation or piecewise constant parameterization).
Simultaneously, the cost functional is approximated as a cost function.
Then, the coefficients of the function approximations are treated
as optimization variables and the problem is reformulated
as a standard nonlinear optimization problem (NLP):
\begin{gather*}
\underset{x}{\min} \ F(x) \\
\text{s.t. } \ c_{i}(x)=0, \quad i\in E, \\
c_{j}(x)\geq 0, \quad j \in I,
\end{gather*}
where $c_{i}$, $i\in E$, and $c_{j}$, $j\in I$,
are the set of equality and inequality constraints, respectively.
In fact, the NLP is easier to solve than the boundary-value problem,
mainly due to the sparsity of the NLP and the many well-known software programs
that can handle it. As a result, the range of problems that
can be solved via direct methods\index{Direct method} is significantly larger
than the range of problems that can be solved via indirect methods. Direct methods
have become so popular these days that many people have written sophisticated
software programs that employ these methods.
Here we present two types of codes/packages: specific solvers
for OC problems and standard NLP solvers used after a discretization process.


\subsubsection{Specific optimal control software}

\subsubsection*{OC-ODE}

\index{OC-ODE}

The \texttt{OC-ODE} \cite{Matthias2009}, \emph{Optimal Control of Ordinary-Differential Equations},
by Matthias Gerdts, is a collection of \texttt{Fortran 77} routines for optimal control problems subject
to ordinary differential equations. It uses an automatic direct discretization method
for the transformation of the OC problem into a finite-dimensional NLP.
\texttt{OC-ODE} includes procedures for numerical adjoint estimation and sensitivity analysis.

\begin{example}
\label{ex_rubeola_OC-ODE}
Considering the same problem of Example~\ref{example_rubeola}, we show the main part
of the code in \texttt{OC-ODE}. The complete code can be found in the website \cite{SofiaSITE}.
The achieved solution is similar to the indirect approach plotted in Figure~\ref{cap2_rubeola},
and therefore is omitted.
\scriptsize{
\begin{verbatim}
c     Call to OC-ODE
c      OPEN( INFO(9),FILE='OUT',STATUS='UNKNOWN')
      CALL OCODE( T, XL, XU, UL, UU, P, G, BC,
     +     TOL, TAUU, TAUX, LIW, LRW, IRES,
     +     IREALTIME, NREALTIME, HREALTIME,
     +     IADJOINT, RWADJ, LRWADJ, IWADJ, LIWADJ, .FALSE.,
     +     MERIT,IUPDATE,LENACTIVE,ACTIVE,IPARAM,PARAM,
     +     DIM,INFO,IWORK,RWORK,SOL,NVAR,IUSER,USER)
      PRINT*,'Ausgabe der Loesung: NVAR=',NVAR
      WRITE(*,'(E30.16)') (SOL(I),I=1,NVAR)
c      CLOSE(INFO(9))
c      READ(*,*)
      END
c-------------------------------------------------------------------------
c     Objective Function
c-------------------------------------------------------------------------
      SUBROUTINE OBJ( X0, XF, TF, P, V, IUSER, USER )
      IMPLICIT NONE
      INTEGER IUSER(*)
      DOUBLEPRECISION X0(*),XF(*),TF,P(*),V,USER(*)
      V = XF(5)
      RETURN
      END
c-------------------------------------------------------------------------
c     Differential Equation
c-------------------------------------------------------------------------
 SUBROUTINE DAE( T, X, XP, U, P, F, IFLAG, IUSER, USER )
      IMPLICIT NONE
      INTEGER IFLAG,IUSER(*)
      DOUBLEPRECISION T,X(*),XP(*),U(*),P(*),F(*),USER(*)
c     INTEGER NONE
      DOUBLEPRECISION B, E, G, P, Q, BETA, A

      B = 0.012D0
      E = 36.5D0
      G = 30.417D0
      P = 0.65D0
      Q = 0.65D0
      BETA = 527.59D0
      A = 100.0D0

      F(1) = B-B*(P*X(2)+Q*X(3))-B*X(1)-BETA*X(1)*X(3)-U(1)*X(1)
      F(2) = B*P*X(2)+BETA*X(1)*X(3)-(E+B)*X(2)
      F(3) = E*X(2)-(G+B)*X(3)
      F(4) = B-B*X(4)
      F(5) = A*X(3))+U(1)**2
      RETURN
      END
\end{verbatim}
}
\end{example}


\subsubsection*{DOTcvp}
\index{DOTcvp}

The \texttt{DOTcvp} \cite{Dotcvp}, \emph{Dynamic Optimization Toolbox with Vector Control Parametrization},
is a dynamic optimization toolbox for \texttt{Matlab}\index{Matlab}. The toolbox provides
an environment for a \texttt{Fortran} compiler to create the '.dll' files of the ODE, Jacobian, and sensitivities.
However, a \texttt{Fortran} compiler has to be installed in the \texttt{Matlab} environment.
The toolbox uses the control vector parametrization approach for the calculation
of the optimal control profiles, giving a piecewise solution for the control.
The OC problem has to be defined in Mayer form\index{Mayer form}.
For solving the NLP, the user can choose several deterministic solvers
--- \texttt{Ipopt}\index{Ipopt}, \texttt{Fmincon}, \texttt{FSQP} ---
or stochastic solvers --- \texttt{DE}, \texttt{SRES}.
The modified \texttt{SUNDIALS} tool \cite{Hindmarsh2005} is used for solving
the IVP and for the gradients and Jacobian automatic generation. Forward integration
of the ODE system is ensured by CVODES, a part of \texttt{SUNDIALS}, which is able to perform
the simultaneous or staggered sensitivity analysis too. The IVP can be solved with the Newton
or Functional iteration module and with the Adams or BDF linear multistep method. Note that the
sensitivity equations are analytically provided and the error control strategy for the sensitivity variables
can be enabled. \texttt{DOTcvp} has a user friendly graphical interface (GUI).

\begin{example}
Considering the same problem of Example~\ref{example_rubeola},
we present here a part of the code used in \texttt{DOTcvp}. The complete
code can be found in the website \cite{SofiaSITE}. The solution,
despite being piecewise continuous, follows the curves plotted in Figure~\ref{cap2_rubeola}.
\tiny{
\begin{verbatim}
% --------------------------------------------------- %
% Settings for IVP (ODEs, sensitivities):
% --------------------------------------------------- %
data.odes.Def_FORTRAN     = {''}; %this option is needed only for FORTRAN parameters definition,
                            e.g. {'double precision k10, k20, ..'}
data.odes.parameters      = {'b=0.012',' e=36.5',' g=30.417',' p=0.65',' q=0.65',' beta=527.59',
                            ' d=0',' phi1=0','phi2=0','A=100 '};
data.odes.Def_MATLAB      = {''}; %this option is needed only for MATLAB parameters definition
data.odes.res(1)          = {'b-b*(p*y(2)+q*y(3))-b*y(1)-beta*y(1)*y(3)-u(1)*y(1)'};
data.odes.res(2)          = {'b*(p*y(2)+q*phi1*y(3))+beta*y(1)*y(3)-(e+b)*y(2)'};
data.odes.res(3)          = {'b*q*phi2*y(3)+e*y(2)-(g+b)*y(3)'};
data.odes.res(4)          = {'b-b*y(4)'};
data.odes.res(5)          = {'A*y(3)+u(1)*u(1)'};
data.odes.black_box       = {'None','1.0','FunctionName'}; %['None'|'Full'],[penalty coefficient
                            for all constraints],...
                            [a black box model function name]
data.odes.ic              = [0.0555 0.0003 0.0004 1 0];
data.odes.NUMs            = size(data.odes.res,2); %number of state variables (y)
data.odes.t0              = 0.0; %initial time
data.odes.tf              = 3; %final time
data.odes.NonlinearSolver = 'Newton'; %['Newton'|'Functional'] /Newton for stiff problems;
                            Functional for non-stiff problems
data.odes.LinearSolver    = 'Dense'; %direct ['Dense'|'Diag'|'Band']; iterative
                            ['GMRES'|'BiCGStab'|'TFQMR'] /for the Newton NLS
data.odes.LMM             = 'Adams'; %['Adams'|'BDF'] /Adams for non-stiff problems;
                            BDF for stiff problems
data.odes.MaxNumStep      = 500; %maximum number of steps
data.odes.RelTol          = 1e-007; %IVP relative tolerance level
data.odes.AbsTol          = 1e-007; %IVP absolute tolerance level
data.sens.SensAbsTol      = 1e-007; %absolute tolerance for sensitivity variables
data.sens.SensMethod      = 'Staggered'; %['Staggered'|'Staggered1'|'Simultaneous']
data.sens.SensErrorControl= 'on'; %['on'|'off']
% --------------------------------------------------- %
% NLP definition:
% --------------------------------------------------- %
data.nlp.RHO              = 10; %number of time intervals
data.nlp.problem          = 'min'; %['min'|'max']
data.nlp.J0               = 'y(5)'; %cost function: min-max(cost function)
data.nlp.u0               = [0 ]; %initial value for control values
data.nlp.lb               = [0 ]; %lower bounds for control values
data.nlp.ub               = [0.9]; %upper bounds for control values
data.nlp.p0               = []; %initial values for time-independent parameters
data.nlp.lbp              = []; %lower bounds for time-independent parameters
data.nlp.ubp              = []; %upper bounds for time-independent parameters
data.nlp.solver           = 'IPOPT'; %['FMINCON'|'IPOPT'|'SRES'|'DE'|'ACOMI'|'MISQP'|'MITS']
data.nlp.SolverSettings   = 'None'; %insert the name of the file that contains settings
                            for NLP solver, if does not exist use ['None']
data.nlp.NLPtol           = 1e-005; %NLP tolerance level
data.nlp.GradMethod       = 'FiniteDifference'; %['SensitivityEq'|'FiniteDifference'|'None']
data.nlp.MaxIter          = 1000; %Maximum number of iterations
data.nlp.MaxCPUTime       = 60*60*0.25; %Maximum CPU time of the optimization
                            (60*60*0.25) = 15 minutes
data.nlp.approximation    = 'PWC'; %['PWC'|'PWL'] PWL only for: FMINCON & without the
                            free time problem
data.nlp.FreeTime         = 'off'; %['on'|'off'] set 'on' if free time is considered
data.nlp.t0Time           = [data.odes.tf/data.nlp.RHO]; %initial size of the time intervals
data.nlp.lbTime           = 0.01; %lower bound of the time intervals
data.nlp.ubTime           = data.odes.tf; %upper bound of the time intervals
data.nlp.NUMc             = size(data.nlp.u0,2); %number of control variables (u)
data.nlp.NUMi             = 0; %number of integer variables (u) taken from the last
                            control variables,
if not equal to 0 you need to use some MINLP solver ['ACOMI'|'MISQP'|'MITS']
data.nlp.NUMp             = size(data.nlp.p0,2); %number of time-independent parameters (p)
\end{verbatim}
}
\normalsize
\end{example}


\subsubsection*{Muscod-II}
\index{Muscod-II}

In NEOS\index{NEOS platform} platform \cite{NEOS}, there is a large set
of software packages. NEOS is considered as the state of the art in optimization.
One recent solver is \texttt{Muscod-II} \cite{Muscod} (Multiple Shooting CODe for Optimal Control)
for the solution of mixed integer nonlinear ODE or DAE constrained optimal
control problems in an extended \texttt{AMPL}\index{AMPL} format.
\texttt{AMPL} is a modelling language for mathematical programming
created by Fourer, Gay and Kernighan \cite{AMPL}. The modelling languages organize
and automate the tasks of modelling, which can handle a large volume of data and,
moreover, can be used in machines and independent solvers, allowing the user
to concentrate on the model instead of the methodology to reach the solution.
However, the \texttt{AMPL} modelling language itself does not allow the formulation
of differential equations. Hence, the \texttt{TACO Toolkit} has been designed
to implement a small set of extensions for easy and convenient modeling
of optimal control problems in \texttt{AMPL}, without the need for explicit encoding
of discretization schemes. Both the \texttt{TACO Toolkit} and the
NEOS interface to \texttt{Muscod-II} are still under development.

\begin{example}
\ \
\scriptsize{
\begin{verbatim}
include OptimalControl.mod;
var t ;									
var x1, >=0 <=1;		
var x2, >=0 <=1;
var x3, >=0 <=1;
var x4, >=0 <=1;		
var u >=0, <=0.9 suffix type "u0";

minimize
cost: integral (A*x3+u^2,3);	

subject to
	c1: diff(x1,t) = b-b*(p*x2+q*x3)-b*x1-beta*x1*x3-u*x1;
	c2: diff(x2,t) = b*p*x2+beta*x1*x3-(e+b)*x2;
	c3: diff(x3,t) = e*x2-(g+b)*x3;
	c4: diff(x4,t) = b-b*x4;

\end{verbatim}
}
\normalsize
\end{example}


\subsubsection{Nonlinear optimization software}

The three nonlinear optimization software packages presented here,
were used through the NEOS platform with codes formulated in \texttt{AMPL}\index{AMPL}.


\subsubsection*{Ipopt}

The \texttt{Ipopt} \cite{Ipopt}, \emph{Interior Point OPTimizer},\index{Ipopt}
is a software package for large-scale nonlinear optimization.
It is written in \texttt{Fortran} and \texttt{C}. \texttt{Ipopt} implements
a primal-dual interior point method and uses a line search strategy based on the filter method.
\texttt{Ipopt} can be used from various modeling environments.
It is designed to exploit 1st and 2nd derivative information,
if provided, usually via automatic differentiation routines in modeling
environments such as \texttt{AMPL}\index{AMPL}. If no Hessians are provided,
\texttt{Ipopt} will approximate them using a quasi-Newton method, specifically a BFGS update.

\begin{example}
\label{ex_rubeola_IPOPT}
Continuing with problem of Example~\ref{example_rubeola}, the \texttt{AMPL} code
is here shown for \texttt{Ipopt}. The Euler discretization was selected.
This code can also be implemented in other nonlinear software packages
available in NEOS platform, reason why the code for the next two software packages
will not be shown. The full version can be found on the website \cite{SofiaSITE}.
\scriptsize{
\begin{verbatim}
#### OBJECTIVE FUNCTION ###
minimize cost: fc[N];

#### CONSTRAINTS ###
subject to
	i1: x1[0] = x1_0;
	i2: x2[0] = x2_0;
	i3: x3[0] = x3_0;
	i4: x4[0] = x4_0;
	i5: fc[0] = fc_0;
	
	f1 {i in 0..N-1}: x1[i+1] = x1[i] + (tf/N)*(b-b*(p*x2[i]+q*x3[i])-b*x1[i]
                                -beta*x1[i]*x3[i]-u[i]*x1[i]);
	f2 {i in 0..N-1}: x2[i+1] = x2[i]+(tf/N)*(b*p*x2[i]+beta*x1[i]*x3[i]-(e+b)*x2[i]);
	f3 {i in 0..N-1}: x3[i+1] = x3[i] + (tf/N)*(e*x2[i]-(g+b)*x3[i]);
	f4 {i in 0..N-1}: x4[i+1] = x4[i] + (tf/N)*(b-b*x4[i]);
	f5 {i in 0..N-1}: fc[i+1] = fc[i] + (tf/N)*(A*x3[i]+u[i]^2);
\end{verbatim}
}
\normalsize
\end{example}


\subsubsection*{Knitro}
\index{Knitro}

\texttt{Knitro} \cite{Knitro}, short for ``Nonlinear Interior point Trust Region Optimization'',
was created primarily by Richard Waltz, Jorge Nocedal, Todd Plantenga and Richard Byrd.
It was introduced in 2001 as a derivative of academic research at Northwestern,
and has undergone continual improvement since then.
\texttt{Knitro} is also a software for solving large scale mathematical optimization problems
based mainly on the two Interior Point (IP) methods and one active set algorithm.
\texttt{Knitro} is specialized for nonlinear optimization, but also solves linear programming problems,
quadratic programming problems, and systems of nonlinear equations. The unknowns in these problems
must be continuous variables in continuous functions. However, functions can be convex or nonconvex.
The code also provides a multistart option for promoting the computation of the global minimum.
This software was tested through the NEOS platform\index{NEOS platform}.


\subsubsection*{Snopt}
\index{Snopt}

\texttt{Snopt} \cite{Snopt}, by Philip Gill, Walter Murray and Michael Saunders,
is a software package for solving large-scale optimization problems (linear and nonlinear programs).
It is specially effective for nonlinear problems whose functions and gradients are expensive to evaluate.
The functions should be smooth but do not need to be convex.
\texttt{Snopt} is implemented in \texttt{Fortran 77} and distributed as source code.
It uses the SQP (Sequential Quadratic Programming) philosophy,
with an augmented Lagrangian approach combining a trust region
adapted to handle the bound constraints. \texttt{Snopt} is also available
in NEOS platform\index{NEOS platform}.


\section{Conclusion}

Choosing a method for solving an optimal control problem depends largely on the type of problem to be
solved and the amount of time that can be invested in coding.
An indirect shooting method has the advantage of being simple to understand
and produces highly accurate solutions when it converges \cite{Rao2009}.
The accuracy and robustness of a direct method\index{Direct method}
is highly dependent upon the method used. Nevertheless, it is easier
to formulate highly complex problems in a direct way and standard
NLP solvers can be used, which is an extra advantage. This last feature has the benefit of converging
with poor initial guesses and being extremely computationally efficient since most
of the solvers exploit the sparsity of the derivatives in the constraints and objective function.


\section*{Acknowledgements}

This work was supported by FEDER funds through COMPETE --
Operational Programme Factors of Competitiveness (``Programa
Operacional Factores de Competitividade'') and by Portuguese
funds through the Portuguese Foundation for Science and Technology
(``FCT -- Funda\c{c}\~{a}o para a Ci\^{e}ncia e a Tecnologia''), within
project PEst-C/MAT/UI4106/2011 with COMPETE number
FCOMP-01–0124-FEDER-022690. Rodrigues was also supported
by the Center for Research and Development
in Mathematics and Applications (CIDMA),
Monteiro by the R\&D unit ALGORITMI
and project FCOMP-01–0124-FEDER-022674,
Torres by CIDMA and by the FCT project PTDC/EEI-AUT/1450/2012,
co-financed by FEDER under POFC-QREN with COMPETE reference
FCOMP-01-0124-FEDER-028894.



\label{lastpage-01}

\end{document}